\documentclass[amssymb,reqno,amsfonts,refcheck,12pt,verbatim,righttag]{amsart}

\usepackage{graphicx}
\usepackage{graphics,color}
\usepackage{amssymb,mathrsfs}
\usepackage{graphicx,color,tikz,caption,subcaption}
\usepackage[all]{xy}
\usepackage{mathtools}
\usepackage{enumerate}
\usepackage{amsmath}
\usepackage{bbm}
\usepackage[colorlinks, linkcolor=blue,citecolor=blue, anchorcolor=blue,urlcolor=blue,hypertexnames=false]{hyperref}

\numberwithin{equation}{section} % \renewcommand{\rm}{\normalshape} %
\newtheorem{thm}{Theorem}[section]
\newtheorem{lemma}[thm]{Lemma}
\newtheorem{pro}[thm]{Proposition}

\newcommand{\N}{\mathbb{N}}

\usepackage{float}
\begin{document}
\baselineskip 14pt
\title{Exact Diophantine approximation$\colon$ the simultaneous case in $\mathbb{R}^{2}$}
\author{Bo Tan and Qing-Long Zhou}
\address{ School  of  Mathematics  and  Statistics,
                Huazhong  University  of Science  and  Technology, 430074 Wuhan, PR China}
          \email{tanbo@hust.edu.cn}
\address{School  of  Mathematics  and  Statistics, Wuhan University of Technology, 430070 Wuhan, PR China }
\email{zhouql@whut.edu.cn}
%\thanks{$^{\dag}$ Corresponding author.}

\date{}

\begin{abstract}
We fill a gap in the study of the Hausdorff dimension of the set 
of exact approximation order  considered  by Fregoli [Proc. Amer. Math. Soc. 152 (2024), no. 8, 3177--3182].
\end{abstract}

\maketitle
%{\small \tableofcontents}

\section{Introduction}
Let $n$ be a positive integer and let $\psi\colon \N\to [0,\frac{1}{2})$ be  a  non-negative function, 
the $\psi$-simultaneously-well approximable set $W(n,\psi)$ is defined to be 
$$W(n,\psi):=\Big\{(x_1,\ldots,x_n)\in [0,1]^n\colon \max_{1\le i\le n}|\!|qx_i|\!|<\psi(q) \text{ for infinitely many } q\in \N \Big\},$$
where $|\!|x|\!|:=\min\{|x-m|\colon m\in \mathbb{Z}\}$ denotes the distance from $x\in \mathbb{R}$
to the nearest integer. Assuming the monotonicity of $\psi,$ the classical Khintchine’s theorem \cite{K24} states that the Lebesgue measure $\mathcal{L}(W(1,\psi))=0$ or 1 according as the series 
$\sum_q\psi(q)$ converges or diverges.  Duffin-Schaeffer \cite{DS41} proved that Khintchine's theorem generally fails without the monotonicity condition on $\psi$. More precisely, they constructed a function $\psi$ which is supported on a set of very smooth integers (having a large number of small prime factors), such that $\sum_q\psi(q)$ diverges, but  $W(1,\psi)$ is null. 
Further,  consider the set
$$ {W}^{\ast}(1, \psi):=\Big\{x\in [0,1]\colon  |\!|qx-\gamma|\!|^\ast<\psi(q)\text{ for infinitely many } q\in \N\Big\},$$
where $$|\!|qx-\gamma|\!|^\ast:=\min_{\text{gcd}(p,q)=1}|qx-p|.$$
{\bf{Duffin-Schaeffer conjecture}} claimed that$\colon$ for any $\psi\colon \N\to [0,\frac{1}{2}),$ 
\begin{equation*}
\mathcal{L}(W^{\ast}(1,\psi))=\begin{cases}
   0   & \text{if $\sum_{q=1}^{\infty}\frac{\phi(q)}{q}\psi(q)<\infty$}, \\
   ~&\\
    1  & \text{if $\sum_{q=1}^{\infty}\frac{\phi(q)}{q}\psi(q)=\infty$},
\end{cases}
\end{equation*}
where $\phi$ is the Euler's totient function.
This conjecture animated a great deal of research until it was finally proved in a breakthrough of Koukoulopoulos-Maynard \cite{KM20}.

Jarn\'{i}k Theorem \cite{J31} shows, under the monotonicity of $\psi,$  that
$$\dim_{\rm H}W(1,\psi)=\frac{2}{\lambda+1}, \ \ \text{where } \lambda:=\lim_{q\to\infty}\frac{-\log \psi(q)}{\log q}.\footnote{Jarn\'{i}k Theorem still holds with `$\lim$' replaced by `$\liminf$'. }$$
It is worth mentioning that Jarn\'{i}k Theorem can be deduced by combining Khintchine's Theorem and the mass transference principle of Beresnevich-Velani \cite{BV06}.  For a general function $\psi,$ the Hausdorff dimension of the set $W(1,\psi)$ was studied extensively by Hinokuma-Shiga \cite{HS96}.

For $n\ge 2,$ Gallagher \cite{G62} proved that$\colon$   for any $\psi\colon \N\to [0,\frac{1}{2}),$ 
\begin{equation*}
\mathcal{L}(W(n,\psi))=\begin{cases}
   0   & \text{if $\sum_{q=1}^{\infty}(\psi(q))^{n}<\infty$}, \\
   ~&\\
    1  & \text{if $\sum_{q=1}^{\infty}(\psi(q))^{n}=\infty$}.
\end{cases}
\end{equation*}
Rynne \cite{R98} considered the mutli-dimensional generalization of $W(1,\psi),$ and proved that 
\begin{equation}\label{W1}
\dim_{\rm H}W(n,\psi)=\frac{n+1}{\lambda+1}.
\end{equation}

We now turn to discuss the Hausdorff dimension of sets in exact Diophantine approximation. 
The exact  $\psi$-simultaneously-well approximable set $\mathrm{Exact}(n,\psi)$ is defined to be 
\begin{equation*}
\mathrm{Exact}(n,\psi):=W(n,\psi) \backslash \bigcup_{0<c<1}W(n,c\psi).
\end{equation*}
Assume that $q\psi(q)$ is non-increasing and tends to zero at infinity. Bugeaud \cite{B03,B03+} and then Bugeaud and Moreira \cite{BM11} showed that 
\begin{equation}\label{E1}
\dim_{\rm H}\mathrm{Exact}(1,\psi)=\dim_{\rm H}W(1,\psi)=\frac{2}{\lambda+1}.
\end{equation}
Recently,  Fregoli \cite{F24} computed the Hausdorff dimension of $\mathrm{Exact}(n,\psi)$ for $n\ge 3$. The technique used by Fregoli took (\ref{E1}) and lifted it to higher dimensions by a clever observation$\colon$ for ${\bf{x}}=(x_1,\ldots,x_n)\in \mathbb{R}^{n},$  
$$\text{if} ~ x_1\in \mathrm{Exact}(1,\psi) ~\text{and} ~ (x_2,\ldots,x_n)\in W(n-1,\psi), \text{then}~ {\bf{x}}\in \mathrm{Exact}(n,\psi). $$
The dimension result $\frac{n+1}{\lambda+1}$ can then be proven via a Theorem of Jarn\'{i}k on fibers. In this paper, we would like to fill the gap$\colon$ $n=2.$
\begin{thm}\label{thm1}
Assume that $q\psi(q)$ is non-increasing and tends to zero at infinity. Then 
$$\dim_{\rm{H}}\mathrm{Exact}(2,\psi)=\frac{3}{\lambda+1}, \ \ \text{where } \lambda:=\lim_{q\to\infty}\frac{-\log \psi(q)}{\log q}.$$
\end{thm}

Applying the celebrated \emph{Dani correspondence} and the \emph{parametric geometry of numbers} developed by Schmidt-Summerer \cite{SS09}, Roy \cite{R15}, and in particular to the remarkable preprint of Das-Fishman-Simmons-Urba\'{n}ski \cite{DFSU24},  Bandi-De Saxc\'{e} \cite{BS23} computed the Hausdorff dimension of $\mathrm{Exact}(n,\psi)$ for $n\ge 1.$ We only wish to point out that our proof is significantly different.

\section{Proof of Theorem \ref{thm1}}
Before proceeding, we cite two famous results$\colon$ Marstrands' Slicing Lemma and Mass Transference Principle (MTP).

\begin{lemma}[Marstrands' Slicing Lemma, \cite{F14}]\label{MSC}
Let $X, Y$ be two metric space and let $E\subset X\times Y.$ If there is a subset $X_1\subset X$ of Hausdorff dimension $s$ such that for any $x_1\in X_1,$ $\dim_{\rm{H}}\{x_2\in Y\colon (x_1,x_2)\in E\}\ge t,$ then 
$$\dim_{\rm{H}}E\ge s+t.$$ 
\end{lemma}

\begin{lemma}[MTP, \cite{BV06}]\label{MTP}
Let $f$ be a dimension function such that $x^{-1}f(x)$ is monotonic. Then
$$\mathcal{H}^{f}(W^{\ast}(1,\psi))=\mathcal{H}^{f}([0,1])
\text{ if } \sum_{q\in \N} f(\psi(n)/n)\phi(n)=+\infty.$$
\end{lemma}

First, we note that $\mathrm{Exact}(2,\psi)\subset \mathrm W(1,\psi)\times[0,1].$ For $x_1\in \mathrm{Exact}(1,\psi)\subset W(1,\psi),$ set 
$$W(x_1,1,\psi):=\{x_2\in[0,1]\colon (x_1,x_2)\in\mathrm{Exact}(2,\psi)\}.$$
Then Marstrands' Slicing Lemma \ref{MSC} yields that 
$$\dim_{\rm{H}} \mathrm{Exact}(2,\psi)\ge\dim_{\rm{H}}\mathrm{Exact}(1,\psi)+\inf_{x_1\in \mathrm{Exact}(1,\psi)}\dim_{\rm{H}}W(x_1,1,\psi).$$

\begin{pro}\label{low bound}
For $x_1\in\mathrm{Exact}(1,\psi),$ we have $\dim_{\rm{H}}W(x_1,1,\psi)\ge\frac{1}{\lambda+1}.$
\end{pro}
\begin{proof}
Let $\mathcal{Q}(x_1)=\{q\in \N\colon |\!|qx_1|\!|<\psi(q)\}.$ Then $\sharp\mathcal{Q}(x_1)=+\infty.$  Put 
\begin{equation*}
\Psi(q)=\begin{cases}
   \psi(q)   & \text{if  $q\in \mathcal{Q}(x_1)$}, \\
    0  & \text{otherwise}.
\end{cases}
\end{equation*}
Set 
$$W^{\ast}(1,\Psi):=\{x\in[0,1]\colon |\!|qx|\!|^{\ast}<\Psi(q) \text{ for infinitely many } q\in\N\}.$$
It is readily checked that $W^{\ast}(1,\Psi)$ is a subset of $W(x_1,1,\psi).$

For $\varepsilon>0,$ put $s=\frac{1-\varepsilon}{\lambda+1}$. We deduce  that
\begin{align*}
\sum_{q\in \N}\Big(\frac{\Psi(q)}{q}\Big)^{s}\phi(q)=\sum_{q\in\mathcal{Q}(x_1)}\Big(\frac{\psi(q)}{q}\Big)^{s}\phi(q)\ge \sum_{q\in\mathcal{Q}(x_1)}q^{-(1-\varepsilon)}\phi(q)\gg\sum_{q\in\mathcal{Q}(x_1)}\frac{q^{\varepsilon}}{\log q}=+\infty,
\end{align*}
where the first inequality holds by the definition of 
$\lambda;$ the second one follows by the following fact$\colon$
$$\frac{\phi(q)}{q}=\prod\limits_{\substack{p| q, \\ p \text{ is prime number}}}\Big(1-\frac{1}{p}\Big)\ge \prod\limits_{p\le q}\Big(1-\frac{1}{p}\Big)\gg \frac{1}{\log q}\Big(1+O(\frac{1}{\log q})\Big),$$
where the last inequality holds by Mertens' theorem\footnote{Mertens' theorem$\colon$ $\prod\limits_{p\le q}(1-\frac{1}{p})=\frac{e^{-\gamma}}{\log q}(1+O(\frac{1}{\log q})),$ where $\gamma$ is the Euler's constant.}.

By Lemma \ref{MTP}, we obtain 
$$\dim_{\rm{H}}W^{\ast}(1,\Psi)\ge \frac{1}{\lambda+1}.$$
Thus $\dim_{\rm{H}}W(x_1,1,\psi)\ge \frac{1}{\lambda+1}.$
\end{proof}

By (\ref{E1}) and Proposition \ref{low bound}, we get 
$$\dim_{\rm{H}}\mathrm{Exact}(2,\psi)\ge \frac{3}{\lambda+1}.$$
On the other hand, $W(2,\psi)$ is a trivial  covering of $\mathrm{Exact}(2,\psi),$
we have 
$$\dim_{\rm{H}}\mathrm{Exact}(2,\psi)\le \dim_{\rm{H}}W(2,\psi)\overset{\text{(\ref{W1})}}{=}\frac{3}{\lambda+1},$$
as desired.

\subsection*{Acknowledgements} 
This work was supported by NSFC Nos. 12171172, 12201476.

\end{document}